\documentclass[12pt,a4paper]{article}
\usepackage{latexsym,amssymb,amsfonts,amsmath,amsthm,nccmath,float,enumitem,setspace,graphicx}
\usepackage[usenames,dvipsnames]{xcolor}
\usepackage[hidelinks]{hyperref}
\usepackage[margin=2cm]{geometry}

\setlist{topsep=3pt,partopsep=0pt,itemsep=1pt,parsep=0pt}

\hypersetup{
    colorlinks,
    linkcolor={red!80!black},
    citecolor={blue!80!black},
    urlcolor={blue!80!black}
}

\newtheorem{theorem}{Theorem}
\newtheorem{corollary}[theorem]{Corollary}
\newtheorem{lemma}[theorem]{Lemma}
\newtheorem{conjecture}[theorem]{Conjecture}

\theoremstyle{definition}
\newtheorem{definition}[theorem]{Definition}

\renewcommand\emptyset{\varnothing}
\def \deg {{\rm deg}}

\def \leq {\leqslant}
\def \geq {\geqslant}

\def \P {\mathcal{P}}

\def \M {\mathcal{M}}
\def \eps {\varepsilon}
\def \Kfme {K_4^-}

\def \nump#1{\#(#1)}

\def \mod#1{{\:({\rm mod}\ #1)}}

\renewcommand{\geq}{\geqslant}
\renewcommand{\leq}{\leqslant}
\renewcommand{\ge}{\geqslant}
\renewcommand{\le}{\leqslant}

\def\eref#1{$(\ref{#1})$}
\def\sref#1{\S$\ref{#1}$}
\def\lref#1{Lemma~$\ref{#1}$}
\def\tref#1{Theorem~$\ref{#1}$}

\def\fref#1{Figure~$\ref{#1}$}

\let\oldproofname=\proofname
\renewcommand{\proofname}{\rm\bf{\oldproofname}}

\begin{document}
\setstretch{1.1}

\title{\bf Induced path factors of regular graphs}

\author{Saieed Akbari\thanks{Department of Mathematical Sciences,
Sharif University of Technology, Tehran, Iran.
\texttt{s\_akbari@sharif.edu}
} \qquad Daniel Horsley\thanks{School of Mathematics, Monash University, Vic, Australia. \texttt{\{daniel.horsley,ian.wanless\}@monash.edu}}\qquad Ian M. Wanless\footnotemark[2]}

\date{}

\maketitle

\begin{abstract}
An {\em induced path factor} of a graph $G$ is a set of induced paths
in $G$ with the property that every vertex of $G$ is in exactly one of
the paths. The {\em induced path number} $\rho(G)$ of $G$ is the
minimum number of paths in an induced path factor of $G$.  We show
that if $G$ is a connected cubic graph on $n>6$ vertices, then
$\rho(G)\le(n-1)/3$.

Fix an integer $k\ge3$. For each $n$, define $\M_n$ to be the maximum
value of $\rho(G)$ over all connected $k$-regular graphs $G$ on $n$
vertices. As $n\rightarrow\infty$ with $nk$ even, we show that
$c_k=\lim(\M_n/n)$ exists.
We prove that $5/18\le c_3\le1/3$ and $3/7\le c_4\le1/2$ and that
$c_k=\frac12-O(k^{-1})$ for $k\rightarrow\infty$.

\bigskip\noindent
Keywords: Induced path, path factor, covering, regular graph, subcubic graph.

\noindent
Classifications: 05C70, 05C38.
\end{abstract}

\section{Introduction}

We denote the path of order $n$ by $P_n$.
A subgraph $H$ of a graph $G$ is said to be \emph{induced} if, for any
two vertices $x$ and $y$ of $H$, $x$ and $y$ are adjacent in $H$ if
and only if they are adjacent in $G$.
An {\it induced path factor} (IPF) of a graph $G$ is a set of induced
paths in $G$ with the property that every vertex of $G$ is in exactly
one of the paths. We allow paths of any length in an IPF, including
the {\em trivial path $P_1$}.  The {\em induced path number} $\rho(G)$
of $G$ is defined as the minimum number of paths in an IPF of $G$. The
main aim of this paper is to show:

\begin{theorem}\label{t:main}
Suppose that $G$ is a connected cubic graph on $n$ vertices.
If $n\le6$ then $\rho(G)=2$ and if $n>6$ then
$\rho(G)\le(n-1)/3$.
\end{theorem}

Of course, for disconnected cubic graphs the smallest IPF consists of
a minimal IPF of each component. In particular, \tref{t:main}
immediately implies:

\begin{corollary}
  A cubic graph on $n$ vertices has an IPF with at most
  $n/2$ paths. Equality holds if and only if every component is
  isomorphic to the complete graph $K_4$.
\end{corollary}

\tref{t:main} does not generalise to cubic multigraphs. If $n$ is even,
then by adding a parallel edge to every second edge of an $n$-cycle we
get a connected cubic multigraph with no IPF with fewer than $n/2$
paths. \tref{t:main} also does not generalise to subcubic graphs.  To
see this, start with an $(n/4)$-cycle and for every vertex $v$ add a
triangle which is connected to $v$ by one edge, as in
\fref{f:subcubic}. This graph has $n$ vertices but cannot be covered
with fewer than $3n/8$ paths.

\begin{figure}
\begin{centering}
\includegraphics[scale=0.6]{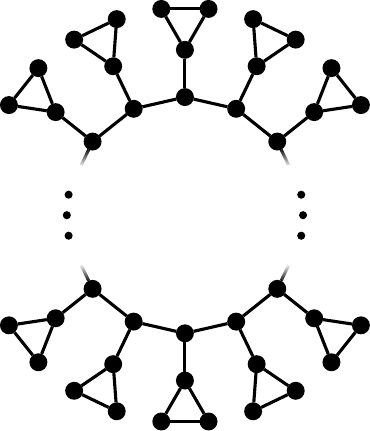}
\caption{\label{f:subcubic}Graph showing $3n/8$ paths may be required
  for a subcubic graph.}
\end{centering}
\end{figure}

It is not clear whether the $n/3-O(1)$ bound in \tref{t:main} can be
improved. However, in \sref{s:higherdeg} we construct a family of
connected cubic graphs $G$ for which $\rho(G)\ge 5n/18+O(1)$. In the same
section we find asymptotic bounds for the maximum value of $\rho(G)/n$
among connected $k$-regular graphs with $n$ vertices, for general $k$.

The concept of induced path number was introduced by Chartrand et
al.~\cite{cha}, who gave the induced path numbers of complete
bipartite graphs, complete binary trees, 2-dimensional meshes,
butterflies and general trees. Broere et al.~\cite{bro} determined the
induced path numbers for complete multipartite graphs.  In \cite{bro},
it was shown that if $G$ is a graph of order $n$, then $\sqrt{n}\leq
\rho(G)+\rho(\overline{G})\leq \lceil \frac{3n}{2}\rceil$, where
$\overline{G}$ denotes the complement of $G$. In \cite{hat}, the best
possible upper and lower bounds for $\rho(G) \rho(\overline{G})$ were
given for two variants: (i) when both $G$ and $\overline{G}$ are
connected and (ii) when neither $G$ nor $\overline{G}$ has isolated
vertices. Pan and Chang \cite{pan} presented an $O(|V|+|E|)$-time
algorithm for finding a minimal IPF on graphs whose blocks are
complete graphs.  Le et al.~\cite{le} proved for general graphs
that it is NP-complete to decide if there is an IPF with a
given number of paths.

Several variants of induced path numbers have been investigated in the
literature.  An IPF in which all paths have order at least two is
called an induced nontrivial path factor (INPF).  In \cite{AMS} the
following was proved:

\begin{theorem}\label{t:INPF}
If $k$ is a positive integer and $G$ is a connected $k$-regular graph
which is not a complete graph of odd order, then $G$ has an INPF.
\end{theorem}

In addition, \cite{AMS} showed that every hamiltonian graph which is
not a complete graph of odd order admits an INPF. Also, if $G$ is a cubic
bipartite graph of order $n\geq 6$, then $G$ has an INPF with size at
most $n/3$.

The {\it path cover number} $\mu(G)$ of $G$ is defined
to be the minimum number of vertex disjoint paths required to cover
the vertices of $G$. Reed \cite{Reed} proved
that $\mu(G) \leq \lceil \frac{n}{9} \rceil$
for any cubic graph of order $n$.
Also, Reed \cite{Reed} conjectured that if $G$ is a 2-connected cubic
graph, then $\mu(G) \leq \lceil \frac{n}{10} \rceil$. This conjecture
was recently proved by Yu \cite{Yu18}.

Magnant and Martin \cite{Magnant1} investigated the path cover number
of regular graphs. They proposed the following interesting conjecture:

\begin{conjecture}
Let $G$ be a $k$-regular graph of order $n$. Then $\mu(G) \leq \frac{n}{k+1}$.
\end{conjecture}

They proved their conjecture for $k\le5$. Kawarabayashi et
al.~\cite{kaw}, proved that every $2$-connected cubic graph of order
at least 6 has a path factor in which the order of each path is at
least 6, and hence it has a path cover using only copies of $P_3$ and
$P_4$.
A subgraph $H$ of a graph $G$ is \emph{spanning} if $H$ has the same
vertex set as $G$.
The {\em minimum leaf number} $ml(G)$ of a connected graph $G$
is the minimum number of leaves among the spanning trees of $G$. In
\cite{goe} it was shown that $\mu(G)+1\leq ml(G)\leq 2\mu(G)$. It was
conjectured that if $G$ is a $2$-connected cubic graph of order $n$,
then $ml(G)\leq \lceil \frac{n}{10}\rceil$.

The structure of this paper is as follows. In the next section we
define terms and notation and prove some basic lemmas about the effect
of simple graph operations on the induced path number. In
\sref{s:23gr} we study IPFs in a certain class of subcubic graphs that
arise when we use induction to find IPFs for cubic graphs.  In
\sref{s:cubic} we prove our main result, \tref{t:main}, drawing on the
results in earlier sections. Finally, in \sref{s:higherdeg} we study
asymptotics for $\rho(G)$ where $G$ is a $k$-regular graph of order
$n$, with $k$ fixed and $n\rightarrow\infty$.

\section{Preliminaries}\label{s:prelim}

Throughout our paper the following notation and terminology will be
used.  When we need to specify the vertices in $P_n$ we will write it
as $[v_1,v_2,\dots,v_n]$, meaning that the edges in the path are
$v_1v_2,v_2v_3,\dots,v_{n-1}v_n$.  Similarly, we use
$(v_1,v_2,\dots,v_n)$ for a cycle of length $n$, with edges
$v_1v_2,v_2v_3,\dots,v_{n-1}v_n,v_nv_1$.  We denote the vertex set and
the edge set of a graph $G$ by $V(G)$ and $E(G)$, respectively. For a
graph $G$ and sets $E \subseteq E(G)$ and $V \subseteq V(G)$, we
denote by $G-E$ the graph obtained from $G$ by deleting the edges in
$E$ and denote by $G - V$ the graph obtained from $G$ by deleting the
vertices in $V$ and all the edges incident on them.
The degree of a vertex $v$ in $G$ will be denoted $\deg_G(v)$. The set
of neighbours of $v$ in $G$ will be denoted $N_G(v)$.  A connected
graph $G$ is said to be $k$-connected if it remains connected whenever
fewer than $k$ vertices are removed. Similarly, $G$ is
$k$-edge-connected if it remains connected whenever fewer than $k$
edges are removed.  A {\it bridge} in a connected graph is an edge
whose removal disconnects the graph.
A graph is called {\it $k$-regular} if each vertex
has degree $k$.  A {\it cubic graph} is a $3$-regular graph and a {\it
  subcubic graph} is a graph with maximum degree at most $3$. A {\it
  k-factor} of a graph is a spanning $k$-regular subgraph of $G$. So a
$2$-factor of $G$ is a disjoint union of cycles of $G$ which covers
all vertices of $G$. A graph is {\it hamiltonian} if it has a
$2$-factor consisting of a single cycle. For distinct positive
integers $a$ and $b$, an {\it $\{a,b\}$-graph} is a graph in which the
degree of each vertex is $a$ or $b$.  The $\{2,3\}$-graphs will play a
major role in our proof of \tref{t:main}. In particular, we will need
$\Kfme$, the graph obtained by removing one edge from the complete
graph $K_4$.
A \emph{block} of a graph is a maximal $2$-connected
subgraph. Throughout the paper when we refer to a block we mean a
block of order at least $3$. Note that because we will only be
concerned with subcubic graphs, their blocks will be vertex disjoint.

While an IPF is formally defined to be a set of paths, an IPF of a
graph $G$ can also be completely specified by giving the set of edges
of $G$ that are in its paths (vertices incident with no edges in the
set are trivial paths). Throughout the paper we will use set
operations to build IPFs from IPFs of subgraphs, as well as to remove
or add edges. Whenever we do so, the IPFs should be considered to be
sets of edges rather than sets of paths. For any IPF $\P$, we use
$\nump{\P}$ to mean the number of paths in $\P$.  When calculating
$\nump{\P}$, it is useful to bear in mind that the number of paths in
an IPF $\P$ of a graph $G$ is always equal to the order of $G$ minus
the number of edges in $\P$. In particular, there is some dependence
on $G$, which will often be implicit.

We start with a lemma showing the effect of two basic operations on graphs.

\begin{lemma}\label{l:glue}
\mbox{ \ }
\begin{itemize}
\item[\textup{(i)}]
If $G'$ is obtained by subdividing an edge of $G$ then $\rho(G')\ge\rho(G)$.
\item[\textup{(ii)}]
If $G$ is obtained from disjoint graphs $A$ and $B$ by identifying a vertex
of $A$ with a vertex of $B$, then
$\rho(G)\ge\rho(A)+\rho(B)-1$.
\end{itemize}
\end{lemma}

\begin{proof}
To show (i), suppose that edge $uv$ of $G$ is subdivided by a new
vertex $w$, thereby forming $G'$. Let $\P'$ be an IPF for $G'$. If a
path in $\P'$ includes both edges $uw$ and $vw$ then replacing those
edges with the edge $uv$ gives an IPF $\P$ for $G$ with
$\nump{\P}=\nump{\P'}$. If a path in $\P'$ includes exactly one of the edges $uw$
and $vw$ then $w$ is the end of the path, so removing the edge incident
with $w$ gives an IPF
$\P$ for $G$ with $\nump{\P}=\nump{\P'}$. Lastly, if $\P'$ includes neither of
the edges $uw$ and $vw$ then $[w]$ is a trivial path in $\P'$. Remove $[w]$
from $\P'$ to get a set of paths in $G$. If each of these $\nump{\P'}-1$
paths is induced, we are done. The only way one of the paths can be
not induced is if it includes both $u$ and $v$. In that case, deleting
one of the edges on the path between $u$ and $v$ creates an IPF of $G$
with $\nump{\P'}$ paths in it.  In all cases, we have succeeded in finding
an IPF of $G$ that has at most $\nump{\P'}$ paths.

We next turn to (ii). Suppose $u\in V(A)$ and $v\in V(B)$ and that $G$
is formed by identifying $u$ with $v$ (for clarity, we will call the
merged vertex $w$). Let $\P$ be an IPF for $G$ with $\nump{\P}=\rho(G)$.
Then $\P$ induces IPFs $\P_A$ and $\P_B$ for $A$ and $B$ respectively.
The path in $\P$ that contains $w$ contributes one path to $\P_A$
and one path to $\P_B$. However, every other
path in $\P$ is wholly within $A$ or within $B$. It follows that
$\nump{\P}=\nump{\P_A}+\nump{\P_B}-1\ge\rho(A)+\rho(B)-1$, and we are done.
\end{proof}

We remark that in both parts of
\lref{l:glue} equality often holds but strict inequality is
possible. For (i), take edges $e_1,e_2,e_3$ that form a $1$-factor in
$K_6$. Let $G=K_6-\{e_1,e_2\}$ and form $G'$ by subdividing $e_3$.
Then $\rho(G)=2$ but $\rho(G')=3$. For (ii), take $A=[a_1,a_2,a_3]$
and $B=[b_1,b_2,b_3]$ and merge $a_2$ with $b_2$ to form $G$. In this
case, $\rho(A)=\rho(B)=1$ but $\rho(G)=3$.

We now introduce the notion of a well-behaved IPF. This definition is
designed for a subsequent application where we will need to ensure
that an IPF of a subcubic graph $H$ is also an IPF of a cubic graph
$G$ that is formed by adding certain edges to $H$.

\begin{definition}
Let $G$ be a subcubic graph, let $S=\{x \in V(G):\deg_G(x) \leq 2\}$,
and let $\P$ be an IPF of $G$. We say that $\P$ is \emph{well-behaved}
(in $G$) if, for each path $P$ of $\P$, we have that either
\begin{itemize}
\item[(i)]
  $V(P) \cap S$ is a subset of the vertices of a single block of $G$; or
\item[(ii)] $P$ contains a subpath $[x,x',y',y]$, where $V(P) \cap S=\{x,y\}$
and $x'y'$ is a bridge of $G$.
\end{itemize}
If the above definition holds with $S$ replaced by
$\{x \in V(G):\deg_G(x) \leq 2\}\setminus R$
for some set of vertices $R$, then we say
that $\P$ is \emph{well-behaved except on $R$}.
\end{definition}

When we say that an IPF is well behaved except on some set $R$, this
does not imply anything about whether the IPF is or is not
well-behaved in the graph overall. The remainder of this section will
be devoted to proving the following lemma, which describes several
surgeries that we will perform on IPFs.

\begin{lemma}\label{l:lift}
Let $G$ and $G'$ be subcubic graphs.
\begin{itemize}
\item[{\rm(i)}] Suppose $G'$ is obtained from $G$ by taking a triangle
  in $G$ on vertex set $\{a,b,c\}$ such that $\deg_G(a)=\deg_G(b)=3$
  and $\deg_G(c)=2$, subdividing $ab$ with a new vertex $d$, and adding
  the edge $cd$. If $G'$ has an IPF $\P'$, then there is an IPF $\P$ of $G$ such that
  $\nump{\P}\le\nump{\P'}$ and a path of $\P$ ends at $c$. Furthermore, $\P \subseteq (\P' \setminus \{bc\}) \cup \{ac\}$ and, if $\P'$ is well behaved,
  then $\P$ is well-behaved except on $\{c\}$.
\item[{\rm(ii)}] Suppose $G'$ is obtained from $G$ by deleting an edge
  $ab$ in $G$ such that $\deg_G(a)=\deg_G(b)=2$, and then
  adding new vertices $\{c,d\}$ and edges $\{ac,ad,bc,bd,cd\}$. If $G'$ has an IPF
  $\P'$, then there is an IPF $\P$ of $G$ such that $\nump{\P} \le\nump{\P'}$ and two distinct
  paths of $\P$ end at $a$ and $b$. Furthermore, $\P \subseteq \P'$ and, if $\P'$ is well behaved,
  then $\P$ is well-behaved except on $\{a,b\}$.
\item[{\rm(iii)}] Suppose $G'$ is obtained from $G$ by deleting a degree
  $2$ vertex $c$ in $G$ such that
  $N_G(c)=\{a,b\}$ and $ab\notin E(G)$, and
  then adding the edge $ab$. If $G'$ has an IPF $\P'$, then there is an IPF $\P$ of $G$ such that $\nump{\P}\le\nump{\P'}+1$ and a path of $\P$ ends at $c$. Furthermore, $\P \subseteq \P' \cup \{ac\}$ and, if $\P'$ is well behaved, then $\P$ is well-behaved except on $\{c\}$.
\end{itemize}
\end{lemma}

If hypothesis (i) holds in \lref{l:lift}, then we say that $G'$ is
obtained from $G$ by \emph{augmenting the triangle on vertex set
$\{a,b,c\}$}. If hypothesis (ii) holds, then we say that $G'$ is
obtained from $G$ by \emph{pasting a $\Kfme$ over $ab$}. If
hypothesis (iii) holds, then we say that $G'$ is obtained from $G$ by
\emph{suppressing the vertex $c$}.

In order to prove \lref{l:lift}, we will require a definition and a further lemma.
Both of these are used only in the proof of \lref{l:lift}.

\begin{definition}
Let $G$ be a subcubic graph, let $\P$ be an IPF of $G$, and let
$(a,b,c,d)$ be a quadruple of vertices of $G$ that induce a $\Kfme$
subgraph that does not contain the edge $ab$. We say that
\emph{$\P$ is standardised for $(a,b,c,d)$} if $c$ is an endpoint of a
path in $\P$ that includes the edge $ac$, and $d$ is an endpoint of a
path in $\P$ that includes the edge $bd$ (note that the two paths must
be distinct).
\end{definition}

\begin{lemma}\label{l:stndrd}
Let $G$ be a subcubic graph and let $(a,b,c,d)$ be vertices
of $G$ that induce a $\Kfme$ subgraph that does not contain the edge
$ab$. If there is an IPF $\P$ of $G$ then there is an IPF $\P^*$ of $G$
such that $\nump{\P^*} \leq \nump{\P}$ and $\P^*$ is standardised for
$(a,b,c,d)$. Moreover, $\P^* \subseteq \P \cup \{ac,bd\}$ and if $\P$
is well-behaved then so is $\P^*$.
\end{lemma}

\begin{proof}
  Suppose that $\P$ is not standardised for $(a,b,c,d)$.
Let $H$ be the subgraph induced by $\{a,b,c,d\}$.
  First suppose
that either $a$ and $b$ are in distinct paths in $\P$ or they are both
in a path that also includes $c$ or $d$. Either way,
$\P^*=(\P \setminus E(H)) \cup \{ac,bd\}$ is an IPF of $G$ with
$\nump{\P^*} \leq \nump{\P}$.

Otherwise $a$ and $b$ are both in a path that includes neither $c$ nor
$d$ and hence there is a vertex $e$ in $G-V(H)$ that is adjacent to $b$
in $G$. Then $\P^*=\big(\P \setminus (E(H) \cup \{be\})\big) \cup
\{ac,bd\}$ is an IPF of $G$ with $\nump{\P^*} \leq \nump{\P}$.

In each case, note that $\P^*$ is standardised for $(a,b,c,d)$ and
$\P^* \subseteq \P \cup \{ac,bd\}$. It remains to justify the claim
that $\P^*$ inherits the well-behaved property from $\P$.  This
follows from the observation that every path $P$ in $\P^*$ has a
subpath $P'$ that includes all vertices of $P$ that have degree $2$ in
$G$, and is such that $P'$ is itself a subpath of a path in $\P$.
\end{proof}

We are now ready to prove Lemma~\ref{l:lift}. It will be useful to note
that by Menger's theorem and the definition of a block, two distinct vertices
are in the same block in a graph $G$ if and only if there is a cycle in $G$
containing both of them.

\begin{proof}[\textbf{\textup{Proof of \lref{l:lift}.}}]
Let $G$ and $G'$ be graphs such that the hypothesis of (i), (ii) or (iii)
holds. We will say we are in case (i), (ii) or (iii) accordingly.
Let $\P'$ be an IPF of $G'$. In cases (i) and (ii), let $\P^*$ be an
IPF of $G'$ such that $\nump{\P^*} \leq \nump{\P'}$, $\P^*$ is
standardised for $(a,b,c,d)$, $\P^* \subseteq \P' \cup \{ac,bd\}$, and
$\P^*$ is well-behaved if $\P'$ is. Such a $\P^*$ exists by
\lref{l:stndrd}. In case (iii), let $\P^*=\P'$.

In case (i), let $\P=\P^*\setminus \{bd\}$. Because $\P^*$ is
standardised for $(a,b,c,d)$, it contains $ac$ and $bd$ but not
$ad$, $bc$ or $cd$.
Thus, $\nump{\P}=\nump{\P^*}$, a path of $\P$ ends at $c$,
and $\P \subseteq(\P'\setminus\{bc\})\cup \{ac\}$. In case (ii), let
$\P=\P^*\setminus\{ac,bd\}$.  Similarly, because $\P^*$ is
standardised for $(a,b,c,d)$, $\P^*$ contains $ac$ and $bd$ but not
$ad$, $bc$ or $cd$. Thus, $\nump{\P}=\nump{\P^*}$, two distinct paths of
$\P$ end at $a$ and $b$, and $\P \subseteq \P'$. In case (iii), let
\[\P=
\left\{
  \begin{array}{ll}
    \P^* & \hbox{if $ab \notin \P^*$} \\
    (\P^*\setminus \{ab\}) \cup \{ac\} & \hbox{if $ab \in \P^*$}
  \end{array}
\right.
\]
So a (possibly trivial) path of $\P$ ends at $c$, and
$\P \subseteq \P' \cup \{ac\}$. Also, $\P$ and $\P^*$ have the same
number of edges, but $G$ has one more vertex than $G'$, so it follows
that $\nump{\P}=\nump{\P^*}+1$.

Now further suppose that $\P'$ is well-behaved and let $R=\{c\}$ in
cases (i) and (iii) and $R=\{a,b\}$ in case (ii). It remains to show
that $\P$ is well-behaved except on $R$. Recall that $\P^*$ is
well-behaved because $\P'$ is. Let
$S=\{x \in V(G):\deg_G(x) \leq 2\}\setminus R$
and note that $S$ is a subset of $S'=\{x \in
V(G'):\deg_{G'}(x) \leq 2\}$. Let $P$ be a path in $\P$ and let $x$
and $y$ be distinct vertices in $V(P) \cap S$. We will complete the
proof by showing that either $x$ and $y$ are in the same block in $G$
or $P$ has a subpath $[x,x',y',y]$ where $V(P) \cap S=\{x,y\}$ and
$x'y'$ is a bridge in $G$.  Given that $x,y \in V(P) \cap S$, it
follows in each case from our definition of $\P$ that $x$ and $y$ were
in the same path $P^*$ in $\P^*$. Thus, because $\P^*$ was
well-behaved and $x,y \in S'$, either $x$ and $y$ were in the same
block in $G'$ or $P^*$ has a subpath $[x,x',y',y]$ where
$V(P^*) \cap S=\{x,y\}$ and $x'y'$ is a bridge in $G'$. If the former
holds, then $x$ and $y$ are in the same block in $G$ (in each case the existence
of a cycle in $G'$ containing $x$ and $y$ implies the existence of a
cycle in $G$ containing $x$ and $y$). So suppose the latter holds. In
cases (i) and (ii), by our definition of $\P$, either $P=P^*$ or $P$
is obtained from $P^*$ by removing an edge in
$\{ac,bd\}$. Furthermore, $[x,x',y',y]$ must be a subpath of $P$
because $x,y \in S$ and $x'y'$ is a bridge in $G'$. In case (iii),
either $P$ is a subpath of $P^*$ (possibly with $P=P^*$) or $P$ is
obtained from a subpath of $P^*$ by adding the edge $ac$. Furthermore,
$[x,x',y',y]$ must be a subpath of $P$ because if $ab$ were an edge in
$[x,x',y',y]$ we would have the contradiction that $x$ and $y$ were in
different paths in $\P$. In each case, the fact that $x'y'$ is a
bridge in $G'$ implies it is a bridge in $G$. This establishes that
$\P$ is well-behaved except on $R$.
\end{proof}

\section{Induced path factors of \texorpdfstring{$\mathbf{\{2,3\}}$}{\{2,3\}}-graphs}\label{s:23gr}

A \emph{triangle ring} is a graph formed by taking an $n$-cycle
$(x_1,\ldots,x_n)$ and adding the chords
$\{x_nx_2,x_3x_5,x_6x_8,\ldots,x_{n-3}x_{n-1}\}$ for some integer $n
\geq 6$ such that $n \equiv 0 \mod{3}$.

Further we say a graph is \emph{bad} if it can be obtained from a
triangle ring by choosing some (possibly empty) set $S$ of its edges
such that no edge in $S$ is in a triangle, and for each edge $e \in S$
proceeding as follows: subdivide $e$ with a vertex $x_e$, add a vertex
disjoint copy of any hamiltonian $\{2,3\}$-graph $H_e$ of order $5$, and
add an edge between $x_e$ and a degree $2$ vertex of $H_e$.

Note that every bad graph has order divisible by $3$. We refer to the
largest block of a bad graph as its \emph{hub}. The hub of a bad graph
has order at least $6$ and each of its other blocks has order $5$. A
fact that will prove useful throughout this section is that a graph
cannot be bad if it contains a vertex of degree $2$ that is in a block
of order at least $6$ but is not in a triangle.

The main result of this section is the following.

\begin{theorem}\label{t:23graphwfactor}
Let $G$ be a connected $\{2,3\}$-graph of order $n \geq 7$ containing
a $2$-factor whose cycles each have length at least $5$. Then
$\rho(G)\le n/3$ if $G$ is a bad graph and $\rho(G)\le (n-1)/3$
otherwise.
\end{theorem}

The example in \fref{f:subcubic} shows that the condition about the
existence of the $2$-factor cannot be dropped from
\tref{t:23graphwfactor}.  Also, note that no cubic graph is bad and
hence \tref{t:23graphwfactor} establishes that any cubic graph $G$ on
at least $7$ vertices with an appropriate $2$-factor has
$\rho(G)\le(n-1)/3$.

Our strategy for building an IPF of a cubic graph $G$ will be to
identify a $2$-factor $F$ in $G$, and to discard some (but not all) of
the edges that join distinct cycles in $F$ so that each cycle in $F$
induces a block.  We then stitch together IPFs of these blocks. To
be efficient we need to allow some paths to include vertices from
more than one block. When this happens the edges that we initially
discarded could potentially cause our paths to not be induced in $G$.
However, by demanding that the constituent IPFs are well-behaved, we will
be able to circumvent this concern.

\tref{t:23graphwfactor} will follow with only a little work from
\lref{l:blocktree} below. Most of our effort in this section will be
devoted to proving \lref{l:blocktree}. The proof proceeds by induction
on the number of blocks in $G$, but we will first require a number of
preliminary results.

\begin{lemma}\label{l:blocktree}
Let $G$ be a connected $\{2,3\}$-graph of order $n \geq 6$ such that
each block of $G$ is a hamiltonian graph of order at least $5$
and the vertex sets of these blocks partition $V(G)$. Then
$G$ has a well-behaved IPF with at most $(n-1)/3$ paths if $n \geq 7$
and $G$ is not a bad graph, and $G$ has a well-behaved IPF with at
most $n/3$ paths otherwise.
\end{lemma}

We begin with three lemmas on IPFs of small hamiltonian $\{2,3\}$-graphs.

\begin{lemma}\label{l:hamorder5}
Let $C$ be a hamiltonian $\{2,3\}$-graph of order $5$. For any vertex
$x$ of degree $2$ in $C$, there is an IPF of $C$ with two paths such
that one path ends at $x$ and every other vertex of this path has
degree $3$ in $C$.
\end{lemma}

\begin{proof}
Let $C'$ be a hamilton cycle in $C$.  For our first path, we take a
shortest path from $x$ around $C'$ that includes one vertex of each
chord of $C'$. The second path also follows $C'$, and joins the
vertices not appearing in the first path.
\end{proof}

\begin{lemma}\label{l:hamorder6}
Let $C$ be a hamiltonian $\{2,3\}$-graph of order $6$. Then $C$ has an
IPF with two paths. Furthermore, for any vertex $x$ of degree $2$ in
$C$, there is an IPF of $C$ with two paths such that one path ends at
$x$ and any other vertices on this path have degree $3$ in $C$ with
the possible exception of the vertex adjacent to $x$ in the path.
\end{lemma}

\begin{proof}
If $C$ is cubic, then it is easy to find an IPF with two paths in each
of the two possible cases for $C$.
If $C$ has a vertex of degree $2$, then we use exactly the same strategy
articulated in the proof of \lref{l:hamorder5}.
\end{proof}

\begin{lemma}\label{l:hamorder7}
Let $C$ be a hamiltonian $\{2,3\}$-graph of order $7$. Let $p=3$ if
$C$ is obtained from a triangle ring of order $6$ by subdividing an
edge that is not in a triangle, and let $p=2$ otherwise. For any
vertex $x$ of degree $2$ in $C$, there is an IPF of $C$ with $p$ paths
such that one path ends at $x$.
\end{lemma}

\begin{proof}
Let $(x,x_1,x_2,\ldots,x_6)$ be a hamilton cycle in $C$.
If $\{x_1x_3,x_4x_6,x_2x_5\}\subseteq E(C)$, then we may use
$\big\{[x,x_1,x_2,x_5],\,[x_3,x_4,x_6]\big\}$ as our IPF.
Otherwise, if $\{x_1x_3,x_4x_6\}\subseteq E(C)$, then $p=3$ and we may use
$\big\{[x,x_1],\,[x_2,x_3,x_4],\,[x_5,x_6]\big\}$ as our IPF.
If only one of the edges
$x_1x_3$ and $x_4x_6$ is in $E(C)$, then by symmetry we may assume it is
$x_1x_3$. In that case, we may take
$\big\{[x,x_1,x_2],\,[x_3,x_4,x_5,x_6]\big\}$ as our IPF. Finally, if
$\{x_1x_3,x_4x_6\}\cap E(C)=\emptyset$, then we may use
$\big\{[x,x_1,x_2,x_3],\,[x_4,x_5,x_6]\big\}$ as our IPF.
\end{proof}

From Lemmas~\ref{l:hamorder5},~\ref{l:hamorder6} and
\ref{l:hamorder7}, we can easily prove the following result concerning
small $\{2,3\}$-graphs with two blocks.

\begin{lemma}\label{l:twoblocks}
Let $G$ be a $\{2,3\}$-graph of order $n \leq 12$ consisting of two
hamiltonian blocks, each of order at least $5$, with a bridge between
them. Then either $G$ has a well-behaved IPF with three paths or $G$
is a bad graph of order $12$ with a well-behaved IPF consisting of $4$
paths.
\end{lemma}

\begin{proof}
Let the two hamiltonian blocks of $G$ be $G_1$ and $G_2$. Let
$n_i=|V(G_i)|$ for $i \in \{1,2\}$, and assume $n_1 \geq n_2$ without
loss of generality. Then $(n_1,n_2) \in
\{(5,5),(6,5),(7,5),(6,6)\}$. Let $x_1x_2$ be the bridge in $G$ where
$x_i \in V(G_i)$ for $i \in \{1,2\}$. For $i \in \{1,2\}$ use
\lref{l:hamorder5}, \ref{l:hamorder6} or \ref{l:hamorder7} as
appropriate to create an IPF $\P_i$ of $G_i$ with one path ending at
$x_i$. Then $\P=\P_1\cup \P_2 \cup \{x_1x_2\}$ is an IPF of $G$. If
$\nump{\P_1}=3$ then $G$ is a bad graph, $(n_1,n_2)=(7,5)$ and $\nump{\P}=4$. In
all other cases, $\nump{\P}=3$. If $n_2=5$, then \lref{l:hamorder5} ensures
that each path in $\P$ obeys (i) in the definition of well-behaved. If
$(n_1,n_2)=(6,6)$, then \lref{l:hamorder6} ensures that each path in
$\P$ obeys either (i) or (ii) in the definition of well-behaved.
\end{proof}

We now prove a more general result for hamiltonian $\{2,3\}$-graphs.
Note that triangle rings are the only bad hamiltonian graphs.

\begin{lemma}\label{l:hamiltonian23}
A hamiltonian $\{2,3\}$-graph $G$ of order $n \geq 6$ has
$\rho(G)\le(n-1)/3$ if $n \geq 7$ and $G$ is not a bad graph,
and has $\rho(G)\le n/3$ otherwise.
\end{lemma}

\begin{proof}
  If $n=6$, the result follows from \lref{l:hamorder6}, so assume $n \geq 7$. Let $F$ be a hamilton cycle in $G$. If $G=F$ the result follows easily, so assume $F$ is a proper subgraph of $G$. We can label $F$ as $(x_1,\ldots,x_n)$ such that $x_nx_k$ is a shortest chord of $F$ in $G$ where $k \in \{2,\ldots,\lfloor n/2 \rfloor\}$. If $k=2$, we can further assume that $x_1$ is not adjacent in $G$ to any vertex in $\{x_3,\ldots,x_{\lfloor (n+1)/2 \rfloor}\}$ (if this is not satisfied, reassign the labels $x_2,\ldots,x_n$ in the opposite orientation around $F$). If $k=3$, we can further assume that $x_1x_4\notin E(G)$. (If $x_1x_4 \in E(G)$ but $x_2x_5 \notin E(G)$, then rotate the labels by one position around $F$. If $\{x_1x_4,x_2x_5\} \subseteq E(G)$ then, noting $x_3x_6 \notin E(G)$, rotate the labels by two positions.)

For $k\ge2$, we now construct an IPF $\P$ of $G$ using a greedy algorithm. We add paths one at a time, at each stage taking a path $[x_i,x_{i+1},\ldots,x_j]$ such that $i$ is the smallest element of $\{1,\ldots,n\}$ for which $x_i$ is not already in a path, and $j$ is the largest element of $\{i,\ldots,n\}$ such that $[x_i,x_{i+1},\ldots,x_j]$ is induced in $G$. We will establish the following:
\begin{itemize}
\item[(i)]
  the first path added to $\P$ has at least $4$ vertices and it has exactly $4$ if and only if $x_ax_5 \in E(G)$ for some $a \in \{1,2,3\}$;
\item[(ii)]
  if $\nump{\P} \geq 3$, then the final path added to $\P$ has at least $2$ vertices;
\item[(iii)] other than the first and last paths added, each path
  $[x_i,x_{i+1}\ldots,x_j]$ in $\P$ has at least $3$ vertices and
  has exactly $3$ if and only if $x_{i+1}x_{i+3} \in E(G)$.
\end{itemize}
The properties of our labelling $(x_1,\ldots,x_n)$ ensure that (i) holds (recall in particular that $x_nx_k$ is a shortest chord of $F$). That (iii) is satisfied follows from the fact that, by our greedy algorithm, for each path $[x_i,\ldots,x_j]$ in $\P$, there is a chord from $x_i$ to a vertex in the path added just prior to $[x_i,\ldots,x_j]$. Similarly, because $x_nx_k \in E(G)$ and $x_k$ is in the first path added to $\P$, there is not a path $[x_i,x_{i+1},\ldots,x_{n-1}]$ in $\P$ for any $i \in \{2,\ldots,x_{n-1}\}$ and (ii) follows.

From (i), (ii) and (iii) we see immediately that $\nump{\P} \leq n/3$. If $\nump{\P} \leq (n-1)/3$ or if $G$ is a triangle ring, then the proof is complete, so assume that $\nump{\P} = n/3$ and $G$ is not a triangle ring. In this remaining case we give an alternative construction for an IPF $\P''$ of $G$ that satisfies the conditions of the lemma. Because $\nump{\P} = n/3$ and $n \geq 7$, it must be that $\nump{\P} \geq 3$ and that the first path added to $\P$ has exactly $4$ vertices, the final path has exactly $2$ vertices and each other path has exactly $3$ vertices. So, by (iii), $\{x_6x_8,x_9x_{11}\ldots,x_{n-3}x_{n-1}\} \subseteq E(G)$. Thus $k=2$ because $x_nx_k$ is a shortest chord of $F$ in $G$. Then, from (i) and the properties of our labelling $(x_1,\ldots,x_n)$, we have that $x_nx_2,x_3x_5 \in E(G)$. This establishes that a triangle ring is a subgraph of $G$ (note that the labelling given in the definition of triangle ring matches our labelling of $G$). By assumption $G$ is not a triangle ring and so there must be an edge $x_ax_b$ in $E(G)$ where $a,b \in \{1,4,7,\ldots,n-2\}$.

If $n=9$ then without loss of generality $a=1$, $b=4$ and we may take
$\P''=\{[x_8,x_9,x_1,x_4],$ $[x_2,x_3,x_5,x_6,x_7]\}$ as our IPF. Henceforth
we may assume that $n\ge12$.  Note that $\P'=E(F)\setminus
\{x_ix_{i+1}:i\in\{1,4,7,\ldots,n-2\}\}$ is an IPF of $G$ with
$\nump{\P'}=n/3$. Let $\P''$ be obtained from $\P'$ by removing the edges
$\{x_{a-1}x_a,x_{b-1}x_b\}$ and adding the edges
$\{x_{a-1}x_{a+1},x_{b-1}x_{b+1},x_ax_b\}$ where we consider
subscripts modulo $n$. As $n>9$, it can be seen that $\P''$ is an IPF
of $G$ with $\nump{\P''}=\nump{\P'}-1=(n-3)/3$. This completes the proof.
\end{proof}

We require two more lemmas before we can complete our proof of
\lref{l:blocktree}. Both concern the structure of a putative minimal
counterexample. Note that by \lref{l:hamiltonian23} we know such a
counterexample has at least two blocks, and hence contains a bridge.

\begin{lemma}\label{l:bridgeobvious}
Let $G$ be a counterexample to \lref{l:blocktree} with the minimum
number of blocks. Let $x_1x_2$ be a bridge in $G$ and let $G_1$ and
$G_2$ be the components of $G-\{x_1x_2\}$. Then either
\begin{itemize}
    \item[(i)]
$|V(G_1)|=5$ or $|V(G_2)|=5$; or
    \item[(ii)]
for each $i \in \{1,2\}$, either $|V(G_i)|=6$ or $G_i$ is a bad graph.
\end{itemize}
\end{lemma}

\begin{proof}
Let $n_i=|V(G_i)|$ for $i \in \{1,2\}$, and suppose for a
contradiction that neither (i) nor (ii) holds. Then, without loss of
generality, $n_1 \geq 7$, $G_1$ is not a bad graph, and $n_2 \geq
6$. By induction there is a well-behaved IPF $\P_1$ of $G_1$ with
$\nump{\P_1} \leq (n_1-1)/3$ and a well-behaved IPF $\P_2$ of $G_2$ with
$\nump{\P_2} \leq n_2/3$. Then $\P=\P_1 \cup \P_2$ is a well-behaved IPF of
$G$ with $\nump{\P} \leq (n_1+n_2-1)/3$, contradicting our assumption that
$G$ is a counterexample to \lref{l:blocktree}.
\end{proof}

\begin{lemma}\label{l:bridgeleaf}
Let $G$ be a counterexample to \lref{l:blocktree} with the minimum
number of blocks. Let $x_1x_2$ be a bridge in $G$ and let $G_1$ and
$G_2$ be the components of $G-\{x_1x_2\}$. Then either $G_1$ or $G_2$
is a block of order $5$.
\end{lemma}

\begin{proof}
Let $n_i=|V(G_i)|$ for $i \in \{1,2\}$, and suppose for a
contradiction that $n_1,n_2 \geq 6$. Say $x_1 \in G_1$ and $x_2 \in
G_2$. By \lref{l:bridgeobvious}, for $i \in \{1,2\}$, either $n_i = 6$
or $G_i$ is a bad graph and $n_i \geq 9$. \lref{l:twoblocks}
eliminates the possibility that $n_1=n_2=6$. So we may assume without
loss of generality that $n_1 \neq 6$ and hence $G_1$ is a bad graph
and $n_1 \geq 9$.

Suppose that one of $x_1$ or $x_2$ is in a block $C$ of order $5$. If
$n_2=6$, then $G_2$ is a block of order 6 and hence it must be $x_1$
that is in $C$. Also, we may suppose without loss of generality that
it is $x_1$ that is in $C$ if $G_2$ is a bad graph and $n_2 \geq 9$.
Let $yz$ be the bridge of $G$ such that $y$ is in $C$ and $z$ is in
the hub $H_1$ of $G_1$. Let $G'_1$ and $G'_2$ be the components of
$G-\{yz\}$ where $V(H_1) \subseteq V(G'_1)$. Observe that $H_1$ is a
block of $G_1$, $|V(H_1)| \geq 7$, $\deg_{G'_1}(z)=2$ and $z$ is not
in a triangle in $H_1$. Thus, $|V(G'_1)| \geq 7$ and $G'_1$ is not
bad. Clearly $|V(G'_2)| \geq n_2+5 \geq 11$. Thus $yz$ violates
\lref{l:bridgeobvious}.

From the argument above we may assume that $x_1$ is in the hub $H_1$
of $G_1$ and hence $x_1$ is in a triangle in $H_1$. Furthermore, if
$n_2 \neq 6$ then $x_2$ is in a triangle in the hub of $G_2$. Because
$x_1$ is in a triangle in $H_1$, $H_1-\{x_1\}$ is hamiltonian, so by
induction there is a well-behaved IPF $\P_1$ of $G_1-\{x_1\}$ with
$\nump{\P_1} \leq (n_1-3)/3$ (recall $n_1 \equiv 0 \mod{3}$). If $n_2 \neq
6$, there is a well-behaved IPF $\P_2$ of $G_2-\{x_2\}$ with $\nump{\P_2}
\leq (n_2-3)/3$ by a similar argument. If $n_2 = 6$, use
\lref{l:hamorder6} to take an IPF $\P_2$ of $G_2$ with two paths, one
of which ends at $x_2$. In either case, $\P=\P_1 \cup \P_2 \cup
\{x_1x_2\}$ is a well-behaved IPF of $G$. If $n_2 \neq 6$, $\nump{\P} =
\nump{\P_1}+\nump{\P_2}+1 \leq (n_1+n_2-3)/3$. If $n_2 = 6$, $\nump{\P} = \nump{\P_1}+2
\leq (n_1+n_2-3)/3$. This contradicts our assumption that $G$ is a
counterexample to \lref{l:blocktree}.
\end{proof}

\begin{proof}[\textbf{\textup{Proof of \lref{l:blocktree}}}]
Suppose for a contradiction that $G$ is a counterexample to
\lref{l:blocktree} with the minimum number of blocks, and let
$n=|V(G)|$.
\lref{l:hamiltonian23} establishes that \lref{l:blocktree} holds when
$G$ is hamiltonian, so $G$ has at
least two blocks. Thus, if $n \leq 12$, $G$ must have exactly two blocks
and \lref{l:twoblocks} establishes that \lref{l:blocktree} holds. So we may
further assume that $n \geq 13$.

It follows from \lref{l:bridgeleaf} and the hypotheses of
\lref{l:blocktree} that $G$ consists of a number $t \geq 1$ of
hamiltonian blocks $L_1,\ldots,L_t$ of order $5$, one other
hamiltonian block $C$ of order at least $5$, and bridges
$x_1y_1,\ldots,x_ty_t$ where $x_1,\ldots,x_t \in V(C)$ and $y_i \in
V(L_i)$ for $i \in \{1,\ldots,t\}$. The proof splits into four cases
according to the placement of the vertices $x_1,\ldots,x_t$ in $C$. In
each case we will construct an IPF $\P$ of $G$ that contradicts our
assumption that $G$ is a counterexample to \lref{l:blocktree}.

\textbf{Case 1.} Suppose that there are $i,j \in \{1,\ldots,t\}$ such
that $x_ix_j \in E(C)$. Then $n \geq 15$. Without loss of generality,
$i=1$ and $j=2$. Let $G_0=G-(V(L_1) \cup V(L_2))$. Let
$G'_0$ be the $\{2,3\}$-graph of order $n-8 \geq 7$ obtained from
$G_0$ by pasting a $\Kfme$ over $x_1x_2$, and note that the
block $C'_0$ of $G'_0$ with $x_1,x_2 \in V(C'_0)$ is hamiltonian. So,
by induction, there is a well-behaved IPF $\P'_0$ of $G'_0$ with
$\nump{\P'_0} \leq (n-8)/3$. By applying
\lref{l:lift}(ii) to $\P'_0$ we
obtain an IPF $\P_0$ of $G_0$ with $\nump{\P_0} \leq (n-8)/3$ that has
paths ending at $x_1$ and $x_2$ and is well-behaved except on
$\{x_1,x_2\}$.   Use \lref{l:hamorder5} to take IPFs
$\P_1$ and $\P_2$ of $L_1$ and $L_2$, each with two paths, where one
path of $\P_1$ ends at $y_1$ and one path of $\P_2$ ends at
$y_2$. Then $\P = \P_0 \cup \P_1 \cup \P_2 \cup \{x_1y_1,x_2y_2\}$ is
a well-behaved IPF of $G$ with $\nump{\P}=\nump{\P_0}+2 \leq (n-2)/3$.

\textbf{Case 2.} Suppose that we are not in Case $1$ and that
$|V(C)|=5$. Then $t=2$, because $n \geq 13$ implies $t \geq 2$ and we
would necessarily be in Case~$1$ if $t\geq 3$. So $n=15$. Without loss
of generality, let $(x_1,u,x_2,v,w)$ be a hamilton cycle in $C$. Use
\lref{l:hamorder5} to take IPFs $\P_1$ and $\P_2$ of $L_1$ and $L_2$,
each with two paths, where one path of $\P_1$ ends at $y_1$ and one
path of $\P_2$ ends at $y_2$. Then $\P = \P_1 \cup \P_2 \cup
\{x_1y_1,x_2y_2,ux_1,ux_2,vw\}$ is a well-behaved IPF of $G$ with
$\nump{\P}=4$.

\textbf{Case 3.} Suppose that we are not in Case $1$ or $2$ and that
$x_1$ is in a triangle of $C$. Because we are not in Case $1$ or $2$,
$|V(C)| \geq 6$. Let $G_0=G-(V(L_1) \cup \{x_1\})$, and note
$|V(G_0)|=n-6 \geq 7$. Note that the block $C_0$ of $G_0$ with vertex
set $V(C)\setminus \{x_1\}$ has $|V(C_0)|\geq 5$ and is hamiltonian
because $x_1$ is in a triangle of $C$. Also, $G_0$ is not bad because
$C_0$ contains two degree $2$ vertices that are not in triangles. So
by induction there is a well-behaved IPF $\P_0$ of $G_0$ with $\nump{\P_0}
\leq (n-7)/3$ paths. Use \lref{l:hamorder5} to take an IPF $\P_1$ of
$L_1$ with two paths such that one path ends at $y_1$. Then $\P=\P_0
\cup \P_1 \cup \{x_1y_1\}$ is a well-behaved IPF of $G$ and $\nump{\P} =
\nump{\P_0}+2 \leq (n-1)/3$.

\textbf{Case 4.} Suppose that we are not in Case $1$, $2$ or $3$. Then
$|V(C)| \geq 6$ and $x_1$ is not in a triangle in $C$. Let
$G_0=G-V(L_1)$, and let $G'_0$ be the graph obtained from $G_0$ by
suppressing vertex $x_1$. Note that $|V(G'_0)|=n-6 \geq 7$ and that
the block $C'_0$ of $G'_0$ with vertex set $V(C)\setminus \{x_1\}$ has
$|V(C'_0)|\geq 5$ and is hamiltonian (since hamiltonicity is preserved
by suppressing a vertex of degree 2). So, by induction, $G'_0$ has a
well-behaved IPF $\P'_0$ with $\nump{\P'_0} \leq (n-6-\delta)/3$
paths, where $\delta=0$ if $G'_0$ is bad and $\delta=1$ otherwise. By
applying \lref{l:lift}(iii) to $\P'_0$ we obtain an IPF $\P_0$ of
$G_0$ with $\nump{\P_0}\leq\nump{\P'_0}+1 \leq (n-3-\delta)/3$ that has a
path ending at $x_1$ and is well-behaved except on $\{x_1\}$. Use
\lref{l:hamorder5} to take an IPF $\P_1$ of $L_1$ with two paths such
that one path ends at $y_1$. Then $\P=\P_0 \cup \P_1\cup\{x_1y_1\}$ is
an IPF of $G$ with $\nump{\P}=\nump{\P_0}+1 \leq (n-\delta)/3$.  If
$G$ is bad or $G'_0$ is not bad then we are done. So we may assume
that $G'_0$ is bad and $G$ is not bad.

As $G'_0$ is bad, it must have a hub and that can only be $C'_0$,
since every other block of $G'_0$ has order 5. So $C'_0$ is obtained
from a triangle ring by subdividing some set of edges not in
triangles. Note that $G$ is obtained from $G'_0$ by subdividing some
edge $uv$ with the vertex $x_1$ and adding $L_1$ and the bridge
$x_1y_1$. So $uv$ is in a triangle in $C'_0$, since otherwise $G$ is
bad or we are in the situation handled by Case 1. Each triangle in
$C'_0$ has two edges in the unique hamilton cycle in $C'_0$ and one
edge not in it. We consider two cases according to which kind of edge
$uv$ is.

If $uv$ is not in the hamilton cycle in $C'_0$, then $C-\{x_1\}$ has
order at least 6 and is hamiltonian. Also, $G-(V(L_1)\cup\{x_1\})$ has
$n-6\ge7$ vertices and is not bad, so by induction it has a
well-behaved IPF $\P_2$ with $\nump{\P_2}\le(n-7)/3$.  Now
$\P_2\cup\P_1\cup\{x_1y_1\}$ is a well-behaved IPF with at most
$2+(n-7)/3=(n-1)/3$ paths, as required.

If $uv$ is in the hamilton cycle in $C'_0$, we can suppose without
loss of generality that $\deg_{G'_0}(u)=3$ and
$\deg_{G'_0}(v)=2$. Then $C-\{x_1,v\}$ has order at least 5 and is
hamiltonian. Also, $G-(V(L_1)\cup\{x_1,v\})$ has $n-7\ge6$ vertices,
so by induction it has a well-behaved IPF $\P_2$ with
$\nump{\P_2}\le(n-7)/3$. Now $\P_2\cup\P_1\cup\{vx_1,x_1y_1\}$ is a
well-behaved IPF with at most $2+(n-7)/3=(n-1)/3$ paths, as required.
\end{proof}

\begin{proof}[\textbf{\textup{Proof of \tref{t:23graphwfactor}}}]
If $G$ satisfies the hypotheses of \lref{l:blocktree}, then we can
apply it to complete the proof, so assume otherwise. Of all the
$2$-factors of $G$ whose cycles each have length at least $5$, let $F$
be one with the minimum number of cycles. Our first goal will be to
obtain a graph $G^*$ from $G$ by deleting edges between cycles of $F$
such that $G^*$ satisfies the hypotheses of \lref{l:blocktree} and is
not a bad graph.

Let $S$ be the set of edges of $G$ that are incident with vertices in
two distinct cycles of $F$ and let $S'$ be a maximal subset of $S$
such that $G-S'$ is connected. For each cycle $A$ of $F$ the graph
$G-S'$ has a hamiltonian block with vertex set $V(A)$. Note that $S'$
is nonempty because $G$ does not satisfy the hypotheses of
\lref{l:blocktree}. If $G-S'$ is not a bad graph, let $S^*=S'$ and
$G^*=G-S^*$. Otherwise $G-S'$ is bad and we proceed as follows. Choose
an arbitrary edge $uv \in S'$ and note that without loss of generality
$u$ is in a block $L$ of order $5$ in $G-S'$ and either $v$ is in a
different block of order $5$ in $G-S'$ or $v$ is in a triangle in the
hub of $G-S'$. Let $S^*=(S' \setminus \{uv\}) \cup \{wx\}$ where $wx$
is the unique bridge in $G-S'$ with $w$ in the hub of $G-S'$ and $x
\in V(L)$. Let $G^*=G-S^*$ and note that $G^*$ is not a bad graph
because, in $G^*$, $w$ is a vertex of degree $2$ that is in a block of
order at least $6$ but not in a triangle.

By \lref{l:blocktree}, there is a well-behaved IPF $\P$ of $G^*$ with
at most $(n-1)/3$ paths. We will show that $\P$ is also an IPF of $G$
and so complete the proof. Suppose otherwise that there is an edge $yz
\in S^*$ such that $y$ and $z$ are both vertices in the same path of
$\P$. Note that, in $G^*$, $y$ and $z$ are vertices of degree $2$ and
are in different blocks. Hence, since $\P$ is well-behaved in $G^*$,
it must be that $G^*$ contains a bridge $y'z'$ such that $yy',zz' \in
E(G^*)$. Since $y$ and $z$ are vertices of degree $2$ in $G^*$, $yy'
\in E(F_1)$ and $zz' \in E(F_2)$ for different cycles $F_1$ and $F_2$
of $F$. However, then the $2$-factor obtained from $F$ by replacing
$F_1$ and $F_2$ with a single cycle with edge set $(E(F_1) \cup E(F_2)
\cup \{yz,y'z'\}) \setminus \{yy',zz'\}$ contradicts our choice of $F$.
\end{proof}

\section{Induced path factors of cubic graphs}\label{s:cubic}

In the previous section we saw that \tref{t:main} holds for any cubic
graph containing a $2$-factor whose cycles all have length at least
$5$. Jackson and Yoshimoto \cite{JY09} showed that any $3$-connected
cubic graph on at least $6$ vertices has such a $2$-factor. In this
section we establish \tref{t:main} via contradiction by showing that a
minimal counterexample to it must be $3$-connected. Recall that for
any subcubic graph the connectivity and edge-connectivity are equal.

\begin{lemma}\label{l:bridgeless}
A counterexample to \tref{t:main} of minimum order is $2$-connected.
\end{lemma}

\begin{proof}
Aiming for a contradiction, suppose that $G$ is a counterexample to
\tref{t:main} of minimum order and that $x_1x_2$ is a bridge in
$G$. For $i \in \{1,2\}$, let $G_i$ be the component of $G-\{x_1x_2\}$
containing $x_i$ and let $n_i=|V(G_i)|$. Then $|V(G)|=n_1+n_2$ and,
because $G$ is cubic, $n_i$ is odd and at least $5$ for $i\in\{1,2\}$.

Let $i \in \{1,2\}$. We claim that $G_i$ has an IPF $\P_i$ such that
$\nump{\P_i} \leq (n_i+1)/3$ and one path ends at $x_i$. If $n_i \in\{5,7\}$,
then it is not hard to see that $G_i$ is hamiltonian (note
that $G_i$ can be obtained from a cubic graph of order $n_i-1$ by
subdividing an edge) and so our claim follows by \lref{l:hamorder5} or
\lref{l:hamorder7}. So we may assume that $n_i \geq 9$. Let $G'_i$ be
the cubic graph obtained from $G_i$ by suppressing $x_i$ if it is not
in a triangle in $G_i$ and augmenting the triangle of $G_i$ containing
$x_i$ otherwise. Let $t=1$ if $x_i$ is in a triangle in $G_i$ and let
$t=0$ otherwise. Then $|V(G_i')|=n_i-1+2t$ and hence $8 \leq |V(G'_i)|
\leq n_1+n_2-4$. So, by induction, there is an IPF $\P'_i$ of $G'_i$
with $\nump{\P'_i} \leq (n_i-2+2t)/3$. Thus our claim holds by applying
\lref{l:lift}(i) to $\P'_i$ if $t=1$ and
\lref{l:lift}(iii) to $\P'_i$ if $t=0$.

Then $\P=\P_1 \cup \P_2 \cup \{x_1x_2\}$ is an IPF of $G$ and
$\nump{\P}=\nump{\P_1}+\nump{\P_2}-1 \leq (n_1+n_2-1)/3$. This contradicts our
assumption that $G$ is a counterexample to \tref{t:main}.
\end{proof}

Next we dispose of a particular configuration that would otherwise cause
us problems later.

\begin{lemma}\label{l:k4e}
A counterexample to \tref{t:main} of minimum order does not contain a
copy $G_1$ of $\Kfme$ such that the two vertices of degree $2$ in
$G-V(G_1)$ are nonadjacent in $G$.
\end{lemma}

\begin{proof}
Aiming for a contradiction, suppose that $G$ is a counterexample to
\tref{t:main} of minimum order that contains a copy $G_1$ of $\Kfme$
such that the two vertices of degree $2$ in $G-V(G_1)$ are nonadjacent
in $G$. Let $n=|V(G)|$. By Lemmas~\ref{l:hamiltonian23} and
\ref{l:bridgeless}, $G$ is nonhamiltonian and bridgeless (note that a
cubic graph cannot be bad). So $n \geq 14$, since the only bridgeless
nonhamiltonian cubic graphs with $12$ or fewer vertices are the
Petersen graph and the Tietze graph (see \fref{f:Tietze}) and neither
of these contains a copy of $\Kfme$.

\begin{figure}
\begin{centering}
\includegraphics[scale=0.6]{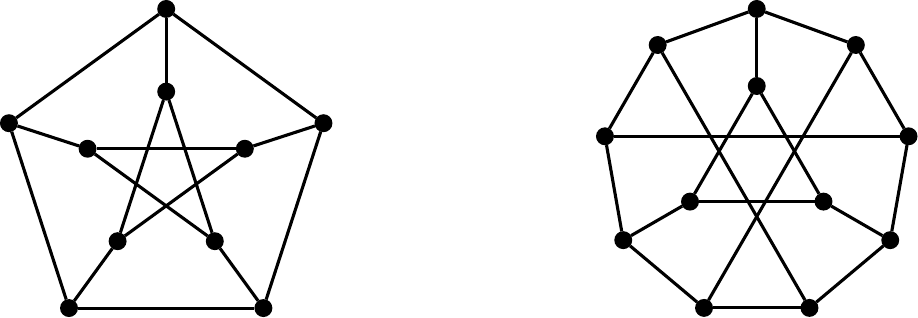}
\caption{\label{f:Tietze}The Petersen graph and the Tietze graph}
\end{centering}
\end{figure}

Let $G_0=G-V(G_1)$. Let $x_0x_1$ and $y_0y_1$ be the two edges of $G$
such that $x_0,y_0 \in V(G_0)$ and $x_1,y_1 \in V(G_1)$, and note that
$x_0y_0 \notin E(G)$ by assumption.  Let $G'_0$ be the graph obtained
from $G_0$ by suppressing $x_0$ if it is not in a triangle in $G_0$
and augmenting the triangle of $G_0$ containing $x_0$ otherwise. In
turn, let $G''_0$ be the cubic graph obtained from $G'_0$ by
suppressing $y_0$ if it is not in a triangle in $G'_0$ and augmenting
the triangle of $G'_0$ containing $y_0$ otherwise.
Let $t'$ (respectively $t''$) be $1$ if $x_0$ (respectively $y_0$)
is in a triangle in $G_0$ and $0$ otherwise.
Let $t=t'+t''$ and note that $|V(G''_0)|=n-6+2t$ and hence
$8 \leq |V(G''_0)| \leq n-2$. So, by induction, there is an IPF
$\P''_0$ of $G''_0$ with $\nump{\P''_0} \leq (n-7+2t)/3$. By applying
\lref{l:lift} to $\P''_0$ (part (i) if $t''=1$ and part (iii) if
$t''=0$) with $c$ chosen to be $y_0$, we can obtain an IPF $\P'_0$ of
$G'_0$ such that a path, $P'$ say, of $\P'_0$ ends at $y_0$ and
$\nump{\P'_0} \leq \nump{\P''_0}+1-t''$. Let $u$ and $v$ be the
neighbours of $x_0$ in $G'_0$ where, without loss of generality,
either $u$ is not in $P'$ or both $u$ and $v$ are in $P'$ and the
subpath of $P'$ from $y_0$ to $v$ does not include $u$. Next we apply
\lref{l:lift} to $\P'_0$ (part (i) if $t'=1$ and part (iii) if $t'=0$)
with $c$ chosen to be $x_0$ and $a$ chosen to be $u$. This produces an
IPF $\P_0$ of $G_0$ such that a path of $\P_0$ ends at $x_0$, $\P_0
\subseteq (\P'_0 \setminus \{x_0v\}) \cup \{x_0u\}$ (note that $x_0v
\notin E(G'_0)$ if $t'=0$), and
\[
\nump{\P_0} \leq \nump{\P'_0}+1-t' \leq \nump{\P''_0}+2-t
\leq (n-1-t)/3 \leq (n-1)/3.
\]
Furthermore, the fact that
$\P_0 \subseteq (\P'_0 \setminus \{x_0v\}) \cup \{x_0u\}$
implies there is a path $P$ of $\P_0$ such that $P$ ends at $y_0$, $E(P) \subseteq E(P') \cup \{x_0u\}$
and $P$ does not contain the edge $x_0v$.
Hence $\{x_0,v\}\nsubseteq V(P)$ and, given our choice of $u$ and $v$,
we have $x_0\notin V(P)$. So distinct paths of $\P_0$ end at $x_0$ and $y_0$.

Let $\P_1$ be an IPF of $G_1$ with two paths such that one ends at
$x_1$ and the other ends at $y_1$. Then
$\P=\P_0 \cup \P_1 \cup \{x_0x_1,y_0y_1\}$ is an IPF of $G$ with
$\nump{\P}=\nump{\P_0} \leq (n-1)/3$, contradicting our assumption
that $G$ is a counterexample to \tref{t:main}.
\end{proof}

We are now ready to prove the connectivity result we want.

\begin{lemma}\label{l:3connected}
A counterexample to \tref{t:main} of minimum order is $3$-connected.
\end{lemma}

\begin{proof}
Aiming for a contradiction, suppose that $G$ is a counterexample to
\tref{t:main} of minimum order and that $G$ is not $3$-connected. By
\lref{l:bridgeless}, $G$ is bridgeless. However, by assumption, there are
two edges $e$ and $f$ whose removal disconnects $G$. Note that $e$ and $f$ are independent since $G$ is cubic and bridgeless.  Thus $G$ is the union of graphs
$G_1$, $G_2$ and $H$ (see \fref{f:ladder}) where
\begin{itemize}
    \item
$V(G_1) \cap V(G_2)=\emptyset$;
    \item
there are vertices $x_1,y_1,x_2,y_2$ such that, for $i \in \{1,2\}$, $V(G_i) \cap V(H)=\{x_i,y_i\}$ and $x_iy_i \notin E(G_i)$;
    \item
for some positive integer $s$, $H$ is the vertex disjoint union of two paths
$[x_1=u_0,\ldots,u_s=x_2]$ and $[y_1=v_0,\ldots,v_s=y_2]$ and a (possibly empty) matching with edge set $\{u_iv_i:1 \leq i \leq s-1\}$;
\end{itemize}
To find this decomposition, we initially take the two paths that define $H$ to be the edges $e$ and $f$, but then extend these paths in both directions until their respective endpoints are not adjacent. For $i \in \{1,2\}$, let $n_i=|V(G_i)|$ and note that because $G$ is cubic $n_i \geq 4$ and $n_i$ is even. Note that $|V(G)|=n_1+n_2+2s-2$.

\begin{figure}
\begin{centering}
\includegraphics[scale=0.8]{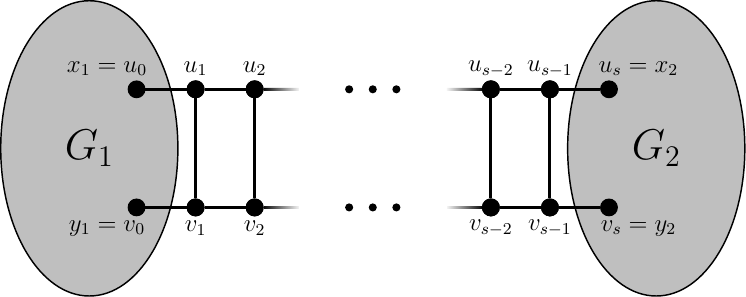}
\caption{\label{f:ladder}Structure of a bridgeless graph when the
removal of 2 edges disconnects it.}
\end{centering}
\end{figure}

Let $i \in \{1,2\}$. We claim that we can find an IPF $\P_i$ of $G_i$
such that $\nump{\P_i} \leq (n_i+2)/3$ and either
$\nump{\P_i}\leq(n_i-1)/3$ or two distinct paths of $\P_i$ end at
$x_i$ and $y_i$. If $n_i\in\{4,6\}$, then $G_i$ must be $\Kfme$ or one
of the three graphs that can be formed by removing an edge from a
cubic graph on $6$ vertices. In each case it is easy to find an IPF of
$G_i$ with two paths where one ends at $x_i$ and the other ends at
$y_i$. If $n_i \geq 8$, then by induction $G_i+\{x_iy_i\}$ has an IPF
$\P'$ with $\nump{\P'} \leq (n_i-1)/3$. We can see that
$\P=\P'\setminus\{x_iy_i\}$ is an IPF of $G_i$ that satisfies the
condition of our claim by considering two cases according to whether
$x_iy_i \in \P'$.

If, for $i \in \{1,2\}$, we have that two distinct paths of $\P_i$ end
at $x_i$ and $y_i$, then $\P=\P_1 \cup \P_2 \cup E([u_0,\ldots,u_s])
\cup E([v_0,\ldots,v_s])$ is an IPF of $G$ with $\nump{\P}=\nump{\P_1}+\nump{\P_2}-2
\leq (n_1+n_2-2)/3$ and $G$ is not a counterexample to
\tref{t:main}. So we may assume without loss of generality that it is
not the case that two distinct paths of $\P_1$ end at $x_1$ and $y_1$
and hence that $\nump{\P_1} \leq (n_1-1)/3$.

If $s \in \{1,2\}$ and $\nump{\P_2} \leq (n_2-1)/3$, then
\[\P=
\left\{
  \begin{array}{ll}
    \P_1 \cup \P_2 & \hbox{if $s=1$} \\
    \P_1 \cup \P_2 \cup \{u_1v_1\} & \hbox{if $s=2$}
  \end{array}
\right.\]
is an IPF of $G$ with $\nump{\P}=\nump{\P_1}+\nump{\P_2}+s-1 \leq (n_1+n_2+3s-5)/3$ and $G$ is not a counterexample to \tref{t:main}.

So we may further assume that either two distinct paths of $\P_2$ end
at $x_2$ and $y_2$ or $\nump{\P_2} \leq (n_2-1)/3$ and $s \geq 3$. In the
former case, let $\P^*_2=\P_2$. In the latter case let $\P^*_2$ be
obtained from $\P_2$ by, for $z \in \{x_2,y_2\}$, if two edges of
$\P_2$ are incident with $z$, deleting one of them. In either case
$\P^*_2$ is an IPF of $G_2$ such that two distinct paths of $\P_2$ end
at $x_2$ and $y_2$ and it can be checked that $\nump{\P^*_2} \leq
(n_2+2s)/3$.

Let $G''_1$ be the cubic graph of order $n_1+4$ obtained from
$G_1$ by adding the vertices $\{u_1,v_1\}$ and edges $\{u_0u_1,v_0v_1,u_1v_1\}$,
and then
pasting a copy $C$ of $\Kfme$ over $u_1v_1$.
By \lref{l:k4e}, we can assume that $s \geq 2$ if $n_2=4$ and so
$G''_1$ has fewer vertices than $G$. So, by induction $G''_1$, has an
IPF $\P''_1$ with $\nump{\P''_1} \leq (n_1+3)/3$ paths. By applying
\lref{l:lift}(ii) to $\P''_1$ we can
obtain an IPF
$\P^*_1$ of the subgraph of $G$ induced by $V(G_1) \cup \{u_1,v_1\}$
such that $\nump{\P^*_1} \leq (n_1+3)/3$ and two distinct paths of $\P_1$
end at $u_1$ and $v_1$.

Then $\P=\P^*_1 \cup \P^*_2 \cup E([u_1,\ldots,u_s]) \cup E([v_1,\ldots,v_s])$
is an IPF $\P$ of $G$ with
\[\nump{\P} \leq \nump{\P^*_1}+\nump{\P^*_2}-2  = (n_1+n_2+2s-3)/3.\]
This contradicts our assumption that $G$ is a counterexample to \tref{t:main}.
\end{proof}

\begin{proof}[\textbf{\textup{Proof of \tref{t:main}}}]
Let $G$ be a counterexample to \tref{t:main} of minimal order $n$. By
\lref{l:3connected}, $G$ is $3$-connected. So, by the theorem of
Jackson and Yoshimoto \cite{JY09}, $G$ has a $2$-factor whose cycles
all have length at least $5$.  \tref{t:23graphwfactor} then implies
that $\rho(G)\le(n-1)/3$ because $G$ is cubic and so cannot be
bad. Hence $G$ is not a counterexample to \tref{t:main} after all, and
this completes the proof of the theorem.
\end{proof}

\section{Induced path factors of regular graphs}\label{s:higherdeg}

In this section, for a fixed integer $k\ge2$, we consider asymptotics for
$\rho(G)$ for $k$-regular graphs $G$ of order $n\rightarrow\infty$.
For each $n$, define $\M_n$ to be the maximum, over all
connected $k$-regular graphs $G$ on $n$ vertices, of $\rho(G)$.
(If no such graphs exist
then we do not consider such $n$ in what follows.) Define
\begin{align*}
\overline{c}_k&=\limsup_{n\rightarrow\infty}\frac{\M_n}{n}\\
\underline{c}_k&=\liminf_{n\rightarrow\infty}\frac{\M_n}{n}\\
{c}_k&=\lim_{n\rightarrow\infty}\frac{\M_n}{n}.
\end{align*}

Our aim for this section is to find bounds for $c_k$. However, first
we must show that it is well defined.

\begin{lemma}\label{l:ckexists}
For each integer $k \geq 2$, $c_k$ exists.
\end{lemma}

\begin{proof}
First we consider the case when $k$ is even. The case $k=2$ is trivial
since $\underline{c}_2=c_2=\overline{c}_2=0$ because every connected
$2$-regular graph has an IPF with $2$ paths. So we may assume that
$k\ge4$.

Fix $\eps>0$. As $k$ is constant, the definition of $\overline{c}_k$
implies that for some suitably large $n_1$ there exists a $k$-regular
graph $G$ of order $n_1$ such that
$\rho(G)\geq(\overline{c}_k-\eps)(n_1+k+4)+1$.

By the Erd\H{o}s-Gallai Theorem \cite{EG60} there exist graphs $F_1$
and $F_2$ of respective orders $k+2$ and $k+3$, with one vertex of
degree $k-2$ and all other vertices of degree $k$. By subdividing an
edge of $G$ with a new vertex that we then identify with the vertex
of degree $k-2$ in $F_1$, we create a $k$-regular graph $G_1$ with
$a=n_1+k+2$ vertices. We use $F_2$ in a similar fashion to create a
$k$-regular graph $G_2$ with $a+1$ vertices. By \lref{l:glue},
$\rho(G_i)\ge\rho(G)\ge(\overline{c}_k-\eps)(a+2)+1$ for $i\in\{1,2\}$.

Now, for any sufficiently large $n$, we make a $k$-regular graph $G'$
of order $n$ as follows. Let $c$ be the least positive integer
satisfying $(k-2)c\equiv(k-2)n+2\pmod{ka-2a+2}$ and let
\[b=\frac{(k-2)(n-c)+2}{ka-2a+2}-c.
\]
Note that $\gcd(k-2,ka-2a+2)=\gcd(k-2,2)=2$ which divides $(k-2)n+2$
so $c$ exists. Also, our choice of $c$ ensures that $b$ is an integer,
and $b>0$ because $n$ is large. We start with $b$ copies of $G_1$ and
$c$ copies of $G_2$ and progressively glue these components together
to form $G'$. Each gluing step takes $k/2$ existing components,
subdivides one edge in each component with a new vertex and
identifies these new vertices. The number of gluing steps is
$(b+c-1)/(k/2-1)$ since each step reduces the number of
components by $k/2-1$. The number of vertices in the resulting graph
$G'$ is
\begin{equation}\label{e:neven}
ba+c(a+1)+\frac{b+c-1}{k/2-1}
=(b+c)(a+2/(k-2))+c-2/(k-2)
=n.
\end{equation}
Also, we started with
$b+c$ components and glued them together, so by \lref{l:glue},
$$\rho(G')\ge b\rho(G_1)+c\rho(G_2)-(b+c)
\ge (\overline{c}_k-\eps)(b+c)(a+2)\ge(\overline{c}_k-\eps)n$$
where the last inequality follows from \eref{e:neven}.
As $n$ was an arbitrary large integer, it follows that
$\underline{c}_k\ge\overline{c}_k-\eps$. But $\eps$ was an
arbitrary positive quantity, so we must have
$\underline{c}_k=\overline{c}_k$, from which it follows that the limit
$c_k$ exists.

It remains to consider the case when $k$ is odd. It works similarly,
but is complicated by the fact that $k$-regular graphs only exist for
even orders.  For some large even integer $a$, we make $G_1$ and $G_2$
of orders $a$ and $a+2$ with $\rho(G_i)\ge(\overline{c}_k-\eps)(a+4)+1$
for $i\in\{1,2\}$. Our gluing steps each involve $k-1$ components.
Two new adjacent vertices are introduced and
$(k-1)/2$ components are glued on each of these two vertices. This reduces
the number of components by $k-2$, so we want $b,c$ to be solutions to
\[
ba+c(a+2)+2\frac{b+c-1}{k-2}=n.
\]
We can take $c$ to be the least positive solution to
$2(k-2)c\equiv(k-2)n+2\pmod{ka-2a+2}$ and let
\[
b=\frac{(k-2)(n-2c)+2}{ka-2a+2}-c.
\]
Note that $\gcd(2(k-2),ka-2a+2)=2$ which divides $(k-2)n+2$ so
$c$ exists. The remainder of the argument mimics the case for even~$k$.
\end{proof}

As mentioned in the proof of \lref{l:ckexists}, $c_2=0$. For larger
$k$ it seems to be a difficult problem to find the exact value of
$c_k$, so instead we look for bounds. Of course, $c_3\le1/3$ by
\tref{t:main}.  We will show that $1/3<c_k\le1/2$ for all $k>3$ and
that $c_k\rightarrow1/2$ as $k\rightarrow\infty$. Note that
$c_k\le1/2$ for all $k$, by \tref{t:INPF}.

In trees all paths are induced. In several of our subsequent results
we use constructions based on perfect $(k-1)$-ary trees.
The root of a $(k-1)$-ary tree has
degree $k-1$, while all other vertices have degree
$k$ or degree $1$ (in the latter case the vertex is a {\it leaf}).
A $(k-1)$-ary tree is {\em perfect} if all its leaves are at the
same distance from the root. In that case the distance from the root
to a leaf is called the {\em height}. We refer to the distance of a
vertex from the root as its {\it depth}. The unique neighbour of a
vertex that has smaller depth than it is its {\it parent}; its other
neighbours have greater depth than it and are its {\it children}.

Our constructions will also often create graphs that contain blocks
that are copies of a complete graph $K_m$ with one edge subdivided.
An IPF of such a block has at least $\lceil m/2 \rceil$ paths, by
\lref{l:glue}.

\begin{lemma}\label{l:covertree}
Let $k\ge3$ and let $T$ be a perfect $(k-1)$-ary tree of
height $h$. Then $\rho(T)=\frac1k\left((k-1)^{h+1}+(-1)^h\right)$.
\end{lemma}

\begin{proof}
Consider a minimal IPF for $T$.  We first argue that without loss of
generality, no path ends at a non-leaf vertex. Suppose this is not
true for a particular IPF. Locate the vertex $v$ of least depth at
which a path ends.  Since $k\ge3$ there is some child $w$ of $v$ which
is not in the path that ends at $v$. We add the edge $vw$. By the
minimality of our IPF, it must include edges $wx$ and $wx'$ where
$x$ and $x'$ are children of $w$. Remove the edge $wx$.  In this
way we create another minimal IPF, and we have reduced the number of
paths that end at the depth of $v$.  So by repeating this process we
can move all ends of paths to the leaves.

Let $a(h)$ be the minimum number of disjoint paths needed to cover a
perfect $(k-1)$-ary tree of height $h$. By the above, we assume
that all paths end at leaves. So by removing the vertices on the path
through the root, we obtain the recurrence
\[
a(h)=1+(k-3)a(h-1)+2(k-2)\sum_{i=0}^{h-2}a(i),
\]
with initial condition $a(0)=1$. We now show that
$a(h)=\frac1k\left((k-1)^{h+1}+(-1)^h\right)$ by induction on $h$. The
formula works for $h=0$. Assuming that it works up to $h-1$, we find that
\begin{align*}
a(h)&=1+\frac1k(k-3)\left((k-1)^{h}+(-1)^{h-1}\right)
+\frac2k(k-2)\sum_{i=0}^{h-2}\left((k-1)^{i+1}+(-1)^i\right)\\
&=\frac1k\Big(k+(k-3)\left((k-1)^{h}+(-1)^{h-1}\right)
+2\left((k-1)^{h}-(k-1)\right)+(2k-4)\chi_{h}\Big)\\
&=\frac1k\Big(k+(k-1)^{h+1}+(-1)^{h-1}(k-3)
-2k+2+(2k-4)\chi_{h}\Big)\\
&=\frac1k\Big((k-1)^{h+1}+(-1)^{h}\Big),
\end{align*}
where $\chi_h=0$ if $h$ is odd and $\chi_h=1$ if $h$ is even.
The result follows.
\end{proof}

We are now ready to give lower bounds on $c_k$ for general $k$. We
treat the cases of odd and even $k$ separately. For each case, we will
construct a family of graphs that are $k$-regular except that the root
vertex has degree less than $k$. It is a simple matter to obtain a
$k$-regular graph by adding a fixed gadget to the root. Doing so
changes the induced path number by $O(1)$, which will be insignificant
for our asymptotics. Hence, for simplicity, we omit details of the
gadgets and pretend that the graphs we build are in fact $k$-regular.

\begin{theorem}\label{t:oddck}
We have $c_3\ge5/18$. For odd $k>3$ we have
\[
c_k\ge \frac12-\frac{3k-4}{k^2(k-1)}=\frac12-O(k^{-2}).
\]
\end{theorem}

\begin{proof}
Start with a perfect $(k-1)$-ary tree $T$ of height $h$. We are
primarily interested in the behaviour of our construction as $h$
becomes large. On each leaf vertex $\ell$, glue $(k-1)/2$ blocks, each
of which is a copy of $K_{k+1}$ with one edge subdivided (the vertex
on the subdivided edge is merged with $\ell$). The resulting graph $G$ is
$k$-regular (except the root), with
\begin{equation}\label{e:vertsodd}
  n=\frac{(k-1)^{h+1}}{k-2}+O(1/k)+\frac12(k-1)^h(k-1)(k+1)
\end{equation}
vertices.

Consider an IPF $\P$ for $G$ and suppose for the moment that $k>3$. If
$\P$ includes the edge from a leaf $\ell$ of $T$ to its parent, remove
this edge from $\P$ and replace it with another edge as follows. Choose
a block $B$ which is glued onto $\ell$, but which contains no neighbour
of $\ell$ in $\P$. The neighbours $u$ and $v$ of $\ell$ in $B$ must both
be ends of paths in $\P$. If they are on different paths in $\P$ then
we simply add the edge $\ell u$ and we are done. Otherwise, $\P$ includes
a path $[u,w,v]$. But $B-\{\ell,u,v,w\}$ is a $K_{k-2}$, and $k-2$ is
odd. Hence $\P$ includes a trivial path, say $[x]$. We replace $[u,v,w]$ and
$[x]$ by $[\ell,u,x]$ and $[v,w]$. In this way, we have not changed the
total number of paths in our IPF, but have removed the edge from $\ell$
to its parent. We repeat this process for all leaves $\ell$.

Now the blocks glued on $\ell$ need $\frac14(k-1)(k+1)-1$ paths
to cover them (the $-1$ is from the path that includes
$\ell$). Applying \lref{l:covertree} to all layers of $T$ except the
last, the number of paths needed to cover $G$ is
\[
(k-1)^h\big((k^2-1)/4-1+1/k\big)+O(1/k).
\]
Combined with \eref{e:vertsodd}, and taking $h\rightarrow\infty$, we find that
\[
c_k\ge\frac{\frac14(k^2-1)-1+1/k}{
\frac{k-1}{k-2}+\frac12(k-1)(k+1)}
=\frac12-\frac{3k-4}{k^2(k-1)}.
\]
Finally, consider the case when $k=3$. Here we apply
\lref{l:covertree} to the whole initial tree $T$.
Then, \lref{l:glue} tells us the effect of gluing a
subdivided $K_4$ onto each of the $(k-1)^h$ leaves.
The conclusion is that $\rho(G)\ge\frac{1}{k}(k-1)^{h+1}+O(1/k)+(k-1)^h$.
Combined with
\eref{e:vertsodd}, and taking $h\rightarrow\infty$, we find that
\[
c_3\ge\frac{\frac{1}{k}(k-1)+1}{
\frac{k-1}{k-2}+\frac12(k-1)(k+1)}
=\frac{2(2k-1)(k-2)}{
k^2(k-1)^2}
=\frac{5}{18}
\]
as claimed.
\end{proof}

\begin{theorem}
We have $c_4\ge3/7$. For even $k\ge4$ we have
\[
c_k\ge\frac{1}{2}-\frac{1}{2k-2}=\frac12-O(k^{-1}).
\]
\end{theorem}

\begin{proof}
Fix an even $k\ge4$. Start with an $h$-cycle $C$ and on each vertex
glue $(k-2)/2$ blocks, each of which is a $K_{k+1}$ with one edge
subdivided. This produces a $k$-regular graph $G$ with
$n=h+h(k+1)(k-2)/2$ vertices. Since $\rho(C)=2$ and
$\rho(K_{k+1})=(k+2)/2$, \lref{l:glue} implies that
$\rho(G)\ge2+h(k-2)k/4$. Taking $h\rightarrow\infty$ gives
$$c_k\ge\frac{(k-2)k/4}{1+(k+1)(k-2)/2}=\frac{1}{2}-\frac{1}{2k-2}$$
as claimed.

While the above argument works for $k=4$, we now provide a separate
construction which gives a stronger result in this case.
Take a perfect 2-tree $T$ of height $h$ and add an edge
between each pair of vertices that are children of the same
parent. Now, for each vertex $\ell$ that was a leaf of $T$, add a copy
of $K_5$ with a subdivided edge (the vertex on the subdivided edge is
identified with $\ell$). The result is a graph with $2^{h+1}-1+2^h5$
vertices that is $4$-regular except for the root.

Suppose we have an IPF for this graph. Consider the two children $v$
and $w$ of a non-leaf vertex $u$ in $T$. If our IPF includes the edge
$vw$ then there must be a path that ends at $u$ (since the neighbours
of $u$, other than $v$ and $w$, are adjacent). Thus we can remove the
edge $vw$ and add the edge $uv$ to get an IPF with the same number of
paths. Repeating this process we obtain an IPF in which no path
includes both children of a vertex in $T$ and hence no path includes
vertices from two distinct subdivided copies of $K_5$. There are $2^h$
subdivided copies of $K_5$, and each requires $3$ paths to cover it.
Thus the graph needs at least $2^h3$ paths in any IPF.  Taking
$h\rightarrow\infty$, it follows that $c_4\ge 3/(2+5)=3/7$, as
claimed.
\end{proof}

Despite trying several alternative constructions, we were unable
to find one for even $k$ which gave an error term matching the one
that we obtained for odd $k$ in \tref{t:oddck}.
Nevertheless, we do not believe there is
a great intrinsic difference between the two cases.
\begin{conjecture}
We have $c_k=1/2-O(k^{-2})$ as $k\rightarrow\infty$.
\end{conjecture}

\subsection*{Acknowledgement}

This research was supported by ARC grants DP150100506 and FT160100048.
Prof.~Akbari is grateful for the hospitality of Monash University,
Australia.  Prof.~Wanless is grateful for the hospitality of IPM,
Tehran.

  \let\oldthebibliography=\thebibliography
  \let\endoldthebibliography=\endthebibliography
  \renewenvironment{thebibliography}[1]{%
    \begin{oldthebibliography}{#1}%
      \setlength{\parskip}{0.3ex plus 0.1ex minus 0.1ex}%
      \setlength{\itemsep}{0.3ex plus 0.1ex minus 0.1ex}%
  }%
  {%
    \end{oldthebibliography}%
  }

\end{document}